\documentclass[english]{amsart}
\usepackage{amsfonts,amssymb,amsmath,amsthm}
\usepackage{url}
\usepackage{enumerate}
\usepackage{stmaryrd}
\newtheorem{thrm}{Theorem}[section]
\newtheorem{lemma}[thrm]{Lemma}

\newtheorem{cor}[thrm]{Corollary}
\theoremstyle{definition}
\newtheorem{definition}[thrm]{Definition}

\newtheorem{example}[thrm]{Example}
\numberwithin{equation}{section}
\usepackage{babel}
\usepackage{pst-all,pst-3dplot,pstricks-add,pst-math,pst-xkey}
\usepackage{graphicx}
\usepackage{latexsym}
\usepackage{accents}
\usepackage{hyperref}
\usepackage{bbold}
\usepackage{epsfig}
\usepackage{subfigure}
\usepackage{skak}
\usepackage{chessboard}
\usepackage{chessfss}
\usepackage{lscape}
\usepackage[all,cmtip]{xy}

\newcommand{\A}{\ensuremath{\mathbb{A}}}

\newcommand{\s}{\ensuremath{\mathbb{S}}}
\newcommand{\U}{\ensuremath{\mathbb{U}}}
\newcommand{\maintain}{\text{Maintain}}
\newcommand{\proviso}{\text{Proviso}}
\newcommand{\GL}{{G^\mathcal{L}}}
\newcommand{\GR}{{G^\mathcal{R}}}
\newcommand{\HL}{H^\mathcal{L}}
\newcommand{\HR}{H^\mathcal{R}}

\newcommand{\Conj}[1]{\overset{\leftrightarrow}{#1}}

\newcommand{\pura}[2]{\langle\emptyset^{#1}\!\mid \!\emptyset^{#2}\rangle}
\newcommand{\cg}[2]{\left\langle #1\!\mid \!#2\right\rangle}

\newcommand{\cgs}[2]{\langle #1\!\mid \!#2\rangle}
\newcommand{\atom}[1]{\emptyset^{#1}}
\newcommand{\mup}{\curlywedgeuparrow}
\newcommand\mown{{\curlyveedownarrow}}

\renewcommand{\ge}{\geqslant}

\newcommand{\su}{\succcurlyeq}

\newcommand{\adr}[1]{\boldsymbol #1}

\newcommand{\Abs}{Absolute Universe}

\author{Urban Larsson}
\address{Dalhousie University, Canada}
\email{urban031@gmail.com}
\author{Richard J. Nowakowski}
\address{Dalhousie University, Canada}
\email{r.nowakowski@dal.ca}
\author{ Carlos P. Santos}
\address{Universidade de Lisboa, Campo Grande, Portugal}
\email{cmfsantos@fc.ul.pt}
\keywords{Keywords: Absolute Universe, Category theory, Combinatorial game space, Dicot mis\`ere, Game comparison}
\subjclass{Primary 91A46, Secondary 18B99}
\begin{document}
\title[When game comparison becomes play]{When game comparison becomes play:\\ Absolutely Categorical Game Theory}
\begin{abstract}
Absolute Universes of combinatorial games, as defined in a recent paper by the same authors, include many standard short normal- mis\`ere- and scoring-play monoids. In this note we show that the class is categorical, by extending Joyal's construction of arrows in normal-play games. Given $G$ and $H$ in an Absolute Universe $\U$, we study instead the Left Provisonal Game $[G, H]$, which is a normal-play game, independently of the particular Absolute Universe, and find that $G\longrightarrow H$ (implying $G\su H$) corresponds to the set of winning strategies for Left playing second in $[G,H]$. By this we define the category ${\bf LNP(\U)}$.
\end{abstract}

\maketitle

\section{Introduction}
This is the study of game comparison in Combinatorial Game Theory (CGT), 
specifically, Combinatorial Game Spaces, and their sub spaces (universes of games). 
The concept of a Combinatorial Game Space allows for a general frame work, which includes many standard classes of terminating games.
 One of the most elegant discoveries of normal-play  CGT, \cite{BerleCG2001--2004}, is that Left wins playing second in the game 
$G$ if and only if $G\geq 0$. Since normal-play games constitute a group structure, this leads to a constructive (subordinate) general 
game comparison, $G\geq H$ if and only if Left wins the game $G - H$ playing second. Joyal \cite{Joyal1977} proved that games, 
under the normal-play convention, form a category 
where $H\longrightarrow G$ if Left wins playing second in $G-H$. That is, Left has good replies 
against  any Right moves $G^R-H$ and $G-H^L$ and so forth. 

More generally, for any winning convention in CGT, game comparison is axiomatized by: Left prefers $G$ to $H$ if, for all games $X$, Left does at least as well in $G+X$ as in $H+X$. Each different winning convention, possibly coupled with other constraints gives a different partial order. 

The authors recently demonstrated \cite{LNS} that there is a set of properties that define \emph{Absolute Universes} and together these properties reduce game comparisons to considering only a certain \emph{Proviso}, and a \emph{Common Normal Part} (corresponding to Theorem~\ref{thrm:ao} in this paper). 
 Except for normal-play, typically Absolute Universes only have a monoid structure (group structure is not common in scoring-play and non-existent in mis\`ere-play), so we cannot use the `inverse' of any game freely.
 It is generally believed that game comparison in normal-play is a special case, which does not apply to other monoids of combinatorial games. 
 
 Here, we construct a normal-play game, called the \emph{Left Provisonal Game}, $[G,H]$ which is essentially playing $G-H$ (as if $H$ were invertible) but where Left's options are restricted by the Proviso, and where the games $G$ and $H$ belong to any \Abs. The previous work \cite{LNS} implies that in any Absolute Universe, the games $G$ and $H$ satisfy $G\su H$ if and only if Left wins the normal-play game $[G, H]$ whenever Right starts (Theorem \ref{thm:dng}). This allows for a construction of arrows, similar to Joyal's, which shows that Absolute Universes are categorical. 

We give the relevant background on Absolute Combinatorial Game Theory \cite{LNS} in Appendix A at the end of this paper. Appendix B contains code for CGsuite 0.7, which `compares' mis\`ere dicot games by, instead, analyzing the Left Provisonal Game. 

\section{Absolute game comparison and the Left Provisional Game }
First we recall the Proviso for a pair of games in a given \Abs\ \cite{LNS}, and we remind the reader that relevant background on outcomes, left-atomic games and so on, is also given in Appendix A.

\begin{definition}[Proviso]\label{def:proviso} Consider an \Abs\ $\U$, and let $G,H\in \U$. The ordered pair of games $[G,H] \in \proviso (\U)$ if
\begin{enumerate}[]
 \item $o_L(G+X)\geqslant o_L(H+X)$ for all left-atomic games $X\in \U$; 
 \item $o_R(G+X)\geqslant o_R(H+X)$ for all right-atomic games $X\in \U$.
\end{enumerate}
\end{definition}
From now onwards, pairs of games in an \Abs\ will combine to another (normal-play) game. 
\begin{definition}[Left Provisonal Game]\label{def:dng} 
Let $\U$ be an Absolute Universe. The \emph{Left Provisonal Game} (LPG) is defined on $\U\times \U$ as follows.
\begin{enumerate}
  \item The positions are ordered pairs of games $[ G, H ]\in \U\times \U$;
  \item The Left options of $[G,H]$ are of the form:
  \begin{enumerate}
    \item $[G^L,H]\in \proviso(\U)$;
    \item $[G,H^R]\in \proviso(\U)$.
  \end{enumerate}
  \item The Right options are games of the form $[G^R,H]$ or $[G,H^L]$;
  \item A player who cannot move loses.
\end{enumerate}
\end{definition}
That is Right cannot move and loses playing first if both $G^\mathcal R$ and $H^\mathcal L$ are empty. 
For Left the situation is more intricate. If for all $G^L$, $[G^L,H]\not\in \proviso(\U)$ and for all $H^R$, $[G,H^R]\not\in \proviso(\U)$, then Left cannot move and loses. Thus, Left Provisonal Game $[G,H]$ is in fact a normal-play game, regardless of $\U$. Using the standard notation of normal-play, thus 
$[G,H]=\{[G,H]^\mathcal L \;|\; [G,H]^\mathcal R\} = \{[G^L,H]\in \proviso(\U), [G,H^R]\in \proviso(\U) \;|\; [G^R,H],[G,H^L]\}$, where $G^L$ ranges over all $G$'s Left options, $H^R$ ranges over $H$'s all Right options, etc.

\begin{definition}[Left's Maintenance]\label{def:common}
Consider an \Abs\ $\U$, and let $G,H\in \U$. The Left Provisional Game  $[ G,H] \in \maintain (\U)\subset \U\times \U$ if, for all Right options $[ G,H] ^R\in [ G,H] ^\mathcal R$, there is a Left option $[ G,H] ^{RL}$, such that  $[ G,H] ^{RL}\in \maintain (\U)$.
\end{definition}

Let us recall the main theorem for comparing games in an Absolute Universe, now stated as an equivalence involving Left Provisonal Games 
(see also Appendix~A).

\begin{thrm}[Basic  order of CGT, \cite{LNS}]\label{thrm:ao}
 Consider an \Abs\ $\U$ and let $G,H\in \U$. Then $G\su H$ if and only if Left Provisonal Game $[G,H]\in \proviso(\U)\cap \maintain(\U).$
\end{thrm}

Analogously:

\begin{thrm}\label{thm:dng}
Let $G,H$ be games in an \Abs\ $\U$. Then $G\su H$ if and only if $[G,H]\in \proviso (\U)$ and the Left Provisonal Game $[G,H]\geq 0$.

\end{thrm}

\begin{proof}
By Theorem~\ref{thrm:ao}, it suffices to prove that $[G,H]\geq 0$ is equivalent with $[G,H]\in \maintain(\U)$. This follows precisely because the inequality means Left wins playing second in normal-play, which is Defininition~\ref{def:dng} (3) combined with the definition of $\maintain(\U)$.
\end{proof}

To the authors' knowledge, in each studied Absolute Universe, $\proviso(\U)$ is constructive, in the sense that the condition in Definition~\ref{def:proviso} can be simplified to compare only (variations of) the outcome of the actual games $G$ and $H$, omitting the potentially infinite class of atomic distinguishing games $X$. For example, in the universe of dicot mis\`ere-play games,  $\proviso(\U)=\{ [G,H] : o(G)\ge o(H)\}$\footnote{Where $o(X)=(o_L(X),o_R(X))\in \{(-1,-1)=R,(-1,+1)=P,(+1,-1)=N,(+1,+1)=L\}$,  $X\in \U$, inducing a partial order of outcomes.}. 

\begin{example}\label{ex:1}
The Proviso simplifies to $o(G)\ge o(H)$ in dicot mis\`ere play, since the only atomic games are the purely atomic ones. 
Take $\U$ as the dicot mis\`{e}re universe and let $G=\mup=\langle0,*\,|\,*\rangle$ 
(``mup'', that means ``mis\`{e}re up'', the simplest dicotic game strictly larger than zero) and $H=0$. 
In the Left Provisional Game  $[ \mup,0] $, Left cannot move to $[ *,0] $, because $P=o(*)\not\ge_\U o(0)=N$ gives that the Proviso is not satisfied. The game tree of the Left Provisonal Game position $[ \mup,0]$ is given in Figure~\ref{fig:gmtree}.
This shows that $[G,H] =\, \uparrow\, > 0$ and thus $G\su H$. Still in dicot mis\`ere, the Left Provisional Game  $[\uparrow , 0]\not \geq 0$, because $R=o(\uparrow)\not \ge o(0)=N$, by the Proviso.

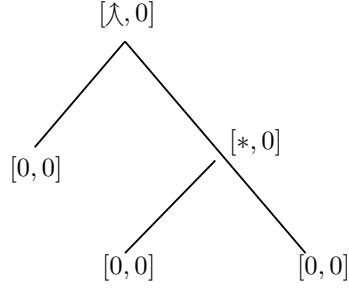
\begin{figure}
\begin{center}
\psset{xunit=0.8cm,yunit=0.8cm,algebraic=true,dotstyle=o,dotsize=2pt 0,linewidth=0.7pt,arrowsize=3pt 2,arrowinset=0.2}
\begin{pspicture*}(1,0.5)(7,6)
\rput[tl](2.68,5.75){$[\mup,0]$}
\psline(3.12,5.08)(1.62,3.32)
\psline(3.12,5.08)(4.62,3.32)
\psline(4.62,3.32)(6.12,1.56)
\psline(4.62,3.1)(3.12,1.56)
\rput[tl](1.2,3.18){$[0,0]$}
\rput[tl](4.85,3.62){$[*,0]$}
\rput[tl](2.74,1.54){$[0,0]$}
\rput[tl](5.98,1.54){$[0,0]{ }$}
\end{pspicture*}
\caption{The game tree of $[\mup,0]$.}\label{fig:gmtree}
\end{center}
\end{figure}
\end{example}

\section{Categories}
Joyal's construction for a category of normal-play games $G$ and $H$ uses that $G\geq H$ if and only if $G-H \geq 0$ if and only if Left has a winning strategy playing second in the game $G-H$ (Left's set of winning strategies is ``the arrow''). In our terminology, this corresponds to the Left Maintenance for the free space of normal-play games. This follows since, for normal-play, the Proviso is implied by the Maintenance part, which is the condition $G\geq H$ in normal-play.\footnote{For normal-play games we use the standard notation for inequality $\geq$, whereas in any other (general) universe we write $\succcurlyeq$.}
We show that each \Abs\ is categorical by extending Joyal's construction to the Left Provisonal Game. 

In a category, the Hom$(H,G)$ is a collection of morphisms that link the object $H$ to the object $G$ in a, for the given structure, specific and meaningful way. The morphisms can be functions but it is not a requirement, as we saw with for example Joyal's winning strategies. The arrows preserve some important property of the given structure, such as ``winning'' in Joyal's example. 
We write $H \longrightarrow G$ 
if Hom$(H,G)$ is not empty (and $\xymatrix{H \ar[r]|f &G}$ if we want to particularize an element $f\in$ Hom$(H,G)$). To have a categorical structure, three properties must hold:

\begin{enumerate}
  \item Identity: $G \longrightarrow G$ for every object $G$;
  \item Composition: given $f\in$ Hom$(H,J)$ and $g\in$ Hom$(J,G)$ there is a  natural 
   composition $g\circ f\in$ Hom$(H,G)$;
   \item Associativity: the defined composition is associative.
\end{enumerate}

We will give a categorical construction based on a calculus of defined \emph{Left Maintenance Strategies} of the LPG. Joyal's and Conway's ``winning'' is merely a consequence of being able to maintain an advantage, specificly, being able to move when it is your turn. By using the LPG rather than the actual games, our ``arrows'' will contain all information of how Left maintains the ability to move, in particular while facing the additional burden of the Proviso part.

\begin{definition}
A \emph{play} in a Left Provisonal Game $X_0=[ G, H]$ is a chain of positions $X_0\leadsto X_1 \leadsto \cdots \leadsto X_n$ 
where the `moves $\leadsto$' correspond to alternating Left and Right (or Right and Left) moves, and where $n\ge 0$.
\end{definition}
Thus, we allow for a play to be perhaps the empty sequence of moves. Of course a play can be defined for any combinatorial game, but we only use it in the  context of Left Provisonal Games. 
\begin{definition}
A \emph{Left Maintenance Strategy} in a given LPG is a play with the following property: consider any stage of the play, where Right is to move; if Right has a move, then Left has a response to this move. 
We write $\mathcal{L}_R(H, G)$ for the set of all Left Maintenance Strategies in the game $[ G,H] $, 
assuming that Right starts, and $\mathcal{L}_L(H,G)$ for all Left Maintenance Strategies, assuming that Left starts.
\end{definition}
Note 1: The reason that we reverse the order of the games in the sets of maintenance strategies is that these will correspond to the homorphisms of the categories, and the order of categorical objects related to ``arrows'' is reversed as compared with the conventions in game theory.

Note 2: Since the LPG is a normal-play game, if you have a maintenance strategy you will eventually win. The particular winning convention of the component games inside the LPG is irrelevant as long as the universe is absolute.

Choosing a strategy $f\in \mathcal{L}_R(H,G)$ is equivalent to choosing a strategy
$f_{G^R}\in \mathcal{L}_L(H,G^R)$ for each position $G^R\in G^\mathcal{R}$ and a strategy
$f_{H^L}\in \mathcal{L}_L(H^L,G)$ for each position $H^L\in H^\mathcal{L}$. Therefore
$$\mathcal{L}_R(H,G)\cong \bigcup_{G^\mathcal{R}} \mathcal{L}_L(H,G^R)\cup \bigcup_{H^\mathcal{L}}
 \mathcal{L}_L(H^L, G).$$

\noindent The concept of a \emph{residual strategy} is crucial in obtaining the composition of morphisms. The fundamental idea is the swivel-chair strategy, using the terminology of \cite{BerleCG2001--2004} (or strategy stealing), see also 
\cite{Joyal1977,CockeCS}.\\

\begin{definition}[Left's Residual Strategy]\label{def:residual}
Given two maintenance strategies $g\in \mathcal{L}_R(J, G)$ and $f\in \mathcal{L}_R(H, J)$, we construct \textit{Left's residual strategy} $g\circledast f $ as follows.

Consider a Right move from $[ G,H] $ to $[ G^R,H] $. We will find a Left's maintenance response, given the candidate morphisms $f$ and $g$.
      
           Set up the two games $[ G, J] $ and $[ J, H] $, corresponding to the maintenance strategies $f$ and $g$ respectively; see the  columns of Figure~\ref{fig:swivelchair}. 
           
           If Left's maintenance response in $[G^R,J]$ is to $[G^{RL},J]$, then adapt this maintenance strategy for the game 
           $[G^R,H]$, as $[G^{RL},H]$. 
           
           If Left's maintenance response in $[G^R, J] $ is to $[G^R,J^R] $ then Left considers instead her maintenance response to the Right move in the game $[J^R, H] $. If this is $[J^R,H^R]$, then her response in $[G^R, H] $ is to $[G^R,H^R]$. If, instead, the response is to some $[J^{RL}, H]$, she swivels back to the game $[ G^R,J^{RL}]$, and finds a response to this Right move, and so on. 
           
           In case $[G^R,J^R] $ is a terminal position, then, because  $[J^R, H] $ is a Right's move in a Left Maintenance Strategy, there must exist a Right move in $H^R$, and so the play will terminate in $[G^R,H^R]$, with a Left win (recall the Left Maintenance Game is normal-play).
           
           In either case, by continuing this idea, because $J$ is finite and because $f$ and $g$ are maintenance strategies, eventually Left's response shifts to either of the forms $[ G^{RL},J^\alpha ] $, with $\alpha={RL\ldots L}$ or $[ J^\alpha,H^R] $ with $\alpha={RL\ldots R}$ (i.e. $\alpha$ is a finite sequence of alternating moves). In the first case, the response in $[ G^R,H]$ will be to $[ G^{RL},H] $ and in the second case it will be to $[ G^R,H^R] $.  Unless this is a terminal position, we may iterate the argument. 
          
The construction of $g\circledast f$ in the case of the Right move $[ G,H^L] $ is analogous.
\end{definition}

By the definition of $f$ and $g$ it is clear that the residual strategy $g\circledast f $ is well defined. As an immediate consequence we get
          
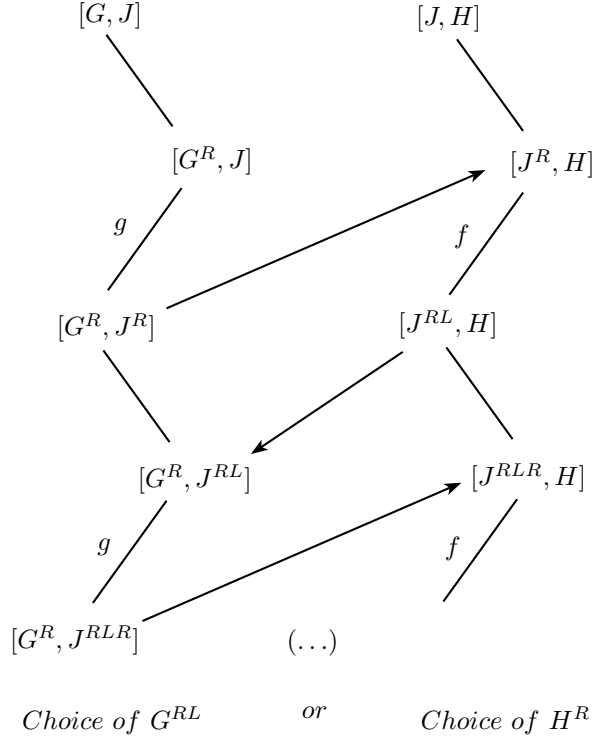
\begin{figure}
\begin{center}
\scalebox{1}{
\psset{xunit=1cm,yunit=1cm,algebraic=true,dotstyle=o,dotsize=3pt 0,linewidth=0.8pt,arrowsize=3pt 2,arrowinset=0.25}
\begin{pspicture*}(-1.5,-5.6)(7.3,5.3)
\rput[tl](-0.28,4.84){$ [G,J] $}
\psline(0.1,4.4)(0.98,3.18)
\rput[tl](0.96,2.92){$ [G^R,J] $}
\psline(1.1,2.36)(0.12,1)
\rput[tl](-0.56,0.7){$ [G^R,J^R] $}
\rput[tl](0.2,1.96){$ g $}
\rput[tl](4.22,4.8){$ [J,H] $}
\psline(4.76,4.34)(5.64,3.12)
\rput[tl](5.46,2.88){$ [J^R,H] $}
\psline{->}(0.9,0.76)(5.14,2.6)
\psline(5.64,2.3)(4.66,0.94)
\rput[tl](4.72,1.88){$ f $}
\rput[tl](3.98,0.74){$ [J^{RL},H] $}
\psline(0.06,0.2)(0.94,-1.02)
\psline(4.62,0.24)(5.5,-0.98)
\rput[tl](0.52,-1.32){$ [G^R,J^{RL}] $}
\psline{->}(4.04,0.18)(2.02,-1.18)
\psline(0.9,-1.8)(-0.08,-3.16)
\rput[tl](-0.02,-2.3){$ g $}
\rput[tl](-1.18,-3.46){$ [G^R,J^{RLR}] $}
\rput[tl](4.94,-1.3){$ [J^{RLR},H] $}
\psline{->}(0.58,-3.4)(4.82,-1.56)
\rput[tl](2.5,-3.5){$ (\ldots) $}
\rput[tl](-1.02,-4.54){$ Choice\,\,of\,\,G^{RL} $}
\rput[tl](2.7,-4.54){$ or $}
\rput[tl](4.28,-4.54){$ Choice\,\,of\,\,H^{R} $}
\psline(5.56,-1.78)(4.58,-3.14)
\rput[tl](4.6,-2.28){$ f $}
\end{pspicture*}
}\vspace{-.5 cm}
\caption{Strategy stealing}\label{fig:swivelchair}
\end{center}
\end{figure}
\begin{lemma}\label{lem:maintain} 
Consider $G,J,H\in \U$ and suppose that $f\in \mathcal{L}_R(H,J)$ and $g\in \mathcal{L}_R(J,G)$ are Left Maintenance Strategies in the games $[ J,H]$ and $[ G,J ]$ respectively. Then the residual strategy $f\circledast g\in \mathcal{L}_R(H,G)$ is a Left Maintenance Strategy in the game $[G, H]$. 
\end{lemma}
\begin{proof}
By the swivel-chair construction in the definition of a residual strategy for the game $[G,H]$, Left has a response to any Right move at each stage of play. Thus $f\circledast g\in \mathcal{L}_R(H,G)$.
\end{proof}
\begin{lemma} \label{lem:asso}
The operator $\circledast$ is associative.
\end{lemma}
\begin{proof}
Given $f\in \mathcal{L}_R(H,J)$, $h\in \mathcal{L}_R(J,W)$, and $g\in \mathcal{L}_R(W,G)$, we construct the composite residual strategy $g\circledast h \circledast f \in \mathcal{L}_R(G,H)$ in analogy with the swivel chair construction in Definition~\ref{def:residual}. Against, say, a Right move from $(G,H)$ to $(G^R,H)$, Left executes the stealing procedure over the strategies $f$, $h$ and $g$, getting, after a finite number of steps, an option $G^{RL}$ or $H^R$. That $g\circledast h \circledast f=g\circledast (h \circledast f)=(g\circledast h) \circledast f$ is then trivial. 
\end{proof}

\begin{definition}[Mimic strategy]
Consider the Left Provisonal Game  position $[G,G]$. We define the \emph{mimic strategy} $m\in \mathcal{L}_R(G,G)$ (or copy-cat) as the strategy where Left replies to $[G^R,G]$ and $[G,G^L]$ with $[G^R,G^R]$ and $[G^L,G^L]$ respectively, and repeats this mimic process during the play.
\end{definition}

\begin{lemma}\label{lem:mimic}
The mimic strategy is a Left Maintenance Strategy.
\end{lemma}
\begin{proof}
In any game of the form $[X,X]$, the proviso is trivially satisfied, so Left has the same options as Right, and, as a required response, can thus imitate each Right move.
\end{proof}
By using maintenance strategies in the Left Provisonal Game as the morphisms, we generalize Joyal's results on categories for normal-play, to any \Abs\ of combinatorial games.
\begin{thrm}\label{thm:cat} Consider an \Abs\ $\U$ and $G,H\in \U$. If $G\su H$, then, the structure $(\mathbb{U}, f, \circ)$, where $ f\in \mathcal{L}_R(H,G)=\mathrm{Hom}(H,G)$, and $g\circ f=g\circledast f$, is categorical.
\end{thrm}
\begin{proof}
By Theorem~\ref{thrm:ao}, $G\su H$ implies $[G,H]\in \maintain(\U)\cap \proviso(\U)$, which in particular implies that, in the LPG, the set of Left's maintenance strategies $\mathcal{L}_R(H,G)$ is nonempty. Moreover, we have seen that the operator is consistent with the residual strategy as composition and the mimic strategy as identity. Indeed, that the following diagram commutes was explained in Lemma~\ref{lem:maintain}.

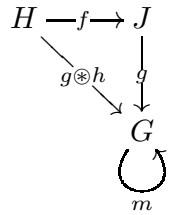
$$\scalebox{1.16}{\xymatrix{H \ar[r]|f \ar[rd]|{g\circledast h} &J \ar[d]|g \\ &G \ar@(dl,dr)_{m}}}\vspace{5 mm}$$
That the defined composition (residual strategy) is associative was explained in Lemma~\ref{lem:asso}.
\end{proof}

For any Absolute Universe $\U$, call this category ${\bf LNP(\U)}$, Left Normal Play over $\U$. We finish off by continuing Example~\ref{ex:1}, the dicot mis\`ere-play application.

\begin{example} We compare the games of rank 2 in the dicot mis\`ere-play universe. 
The Proviso is $o(G)\geqslant o(H)$. The order is given in Figure \ref{fig:dmday2}, the value of the LPG 
$[G,H]$ where $G$ covers $H$ in the partial order is written by the appropriate edge. 
In the picture, the dicot mis\`ere-play game values (literal forms) are ${\uparrow}=\langle 0\,|\,*\rangle$, ${\downarrow}=\langle *\,|\,0\rangle $, 
$\mup=\langle 0,*\,|\,*\rangle $,  $\mown=\langle *\,|\,*,0\rangle $ (``mown''), 
$\mup *=\langle 0,*\,|\,0\rangle $, and ${\mown} =\langle 0\,|\,*,0\rangle $. 
\end{example}

\begin{figure}
\begin{center}
\scalebox{0.75}{
\psset{xunit=0.8cm,yunit=0.8cm,algebraic=true,dotstyle=o,dotsize=5pt 0,linewidth=0.8pt,arrowsize=3pt 2,arrowinset=0.25}
\begin{pspicture*}(-3,-7)(14.7,4)
\psline(1.54,-1.54)(1.54,2.38)
\psline(-2.38,-1.54)(1.54,2.38)
\psline(1.54,2.38)(5.46,-1.54)
\psline(5.46,-1.54)(9.38,-5.46)
\psline(1.54,-1.54)(9.38,-5.46)
\psline(1.54,-5.46)(9.38,-1.54)
\psline(9.38,-1.54)(9.38,2.38)
\psline(5.46,-1.54)(9.38,2.38)
\psline(9.38,2.38)(13.3,-1.54)
\psline(13.3,-1.54)(9.38,-5.46)
\psline(-2.38,-1.54)(1.54,-5.46)
\psline(1.54,-5.46)(5.46,-1.54)
\rput[tl](-2.98,-1.42){\scalebox{2}{$*$}}
\rput[tl](8.1,-0.80){\scalebox{2}{${\uparrow}$}}
\rput[tl](13.5,-1.2){\scalebox{2}{$ 0 $}}
\rput[tl](3.9,-1.25){\scalebox{2}{$ * 2 $}}
\rput[tl](0.9,-1.2){\scalebox{2}{${\downarrow}$}}
\rput[tl](9.76,3.4){\scalebox{2}{$\mup$}}
\rput[tl](2.02,3.4){\scalebox{2}{$\mup$}}
\rput[tl](2.6,3.1){\scalebox{2}{$*$}}
\rput[tl](9.7,-5.82){{\scalebox{2}{$\mown$}}}
\rput[tl](1.96,-5.82){{\scalebox{2}{$\mown$}}}
\rput[tl](2.47,-5.95){{\scalebox{2}{$*$}}}
\rput[tl](-0.92,0.72){$ \uparrow $}
\rput[tl](-0.88,-3.7){$ \uparrow $}
\rput[tl](11.66,-3.6){$ \uparrow $}
\rput[tl](11.72,0.62){$ \uparrow $}
\rput[tl](3.72,-4.6){$ 0 $}
\rput[tl](1.14,0.26){$ 0 $}
\rput[tl](7.04,-4.62){$ 0 $}
\rput[tl](9,0.24){$ 0 $}
\rput[tl](3.92,0.64){$ \{\uparrow\,|\,\downarrow *\} $}
\rput[tl](1.2,-3.64){$ \{\uparrow\,|\,\downarrow *\} $}
\rput[tl](4.2,1.54){$ \{\uparrow\,||\,\{0,\downarrow *\,|\,0,\downarrow *\}\} $}
\rput[tl](7.4,-3.0){$ \{\uparrow\,||\,\{0,\downarrow *\,|\,0,\downarrow *\}\} $}
\begin{scriptsize}
\psdots[dotstyle=*](-2.38,-1.54)
\psdots[dotstyle=*](5.46,-1.54)
\psdots[dotstyle=*,linecolor=darkgray](1.54,-1.54)
\psdots[dotstyle=*,linecolor=darkgray](1.54,-5.46)
\psdots[dotstyle=*,linecolor=darkgray](1.54,2.38)
\psdots[dotstyle=*,linecolor=darkgray](9.38,2.38)
\psdots[dotstyle=*,linecolor=darkgray](9.38,-5.46)
\psdots[dotstyle=*](13.3,-1.54)
\psdots[dotstyle=*,linecolor=darkgray](9.38,-1.54)
\end{scriptsize}
\end{pspicture*}
}
\caption{Games of rank 2 in dicot mis\`ere-play.}\label{fig:dmday2}
\end{center}
\end{figure}
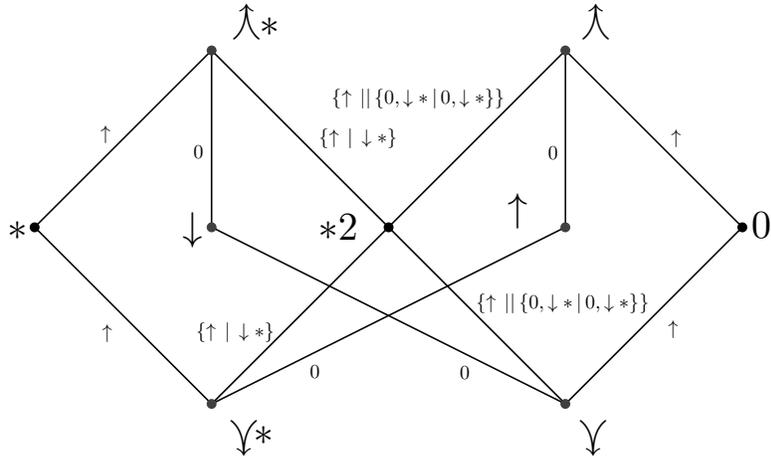
\noindent {\bf Acknowledgement}. We thank Darien DeWolf for suggesting our category's name.
\newpage

\section*{Appendix A}\label{sec:absU}


The following is a shortened introduction to Absolute CGT \cite{LNS}. 
Combinatorial games have two players, usually called \textit{Left} (female) and
 \textit{Right} (male) who move alternately. 
Both players have perfect information, and there are no chance devices. 
Thus these are games of pure strategy with no randomness. 
Combinatorial games are commonly represented by a rooted tree called the game tree. 
The nodes are positions that can be reached 
during the play of the game and the root is the present position. The children of a node 
are all the positions that can be reached 
in one move and these are called \textit{options}. We distinguish between the left-options, 
those positions that Left can reach in one move,
 and the right-options,
 denoted by $\GL$ and $\GR$ respectively. Any game $G$ can be represented by two such lists and we 
 write $G=\cgs{\GL}{\GR}$. Thus, $G$ can be expanded in terms of elements of its \emph{terminal} 
 positions (those positions with no options). 
 The \textit{rank} of a game is the depth of the game tree (see also Definition~\ref{def:recur}). 
 This gives the common proof technique `induction on the options' since 
the depth of the game tree of an option is at least one less than that of the original position 
(we study games without cycles).

Let $(\A,+)$ be a totally ordered, additive group. A 
 terminal position will be of the form 
$\pura{\ell}{r}$ where $\ell, r\in \A$. The intuition, adapted from scoring game theory is that, if Left is to move, then the game is finished, 
and the `score' is $\ell$, and similarly for Right, where the `score' would be $r$. In general, if $G$ is a game with no Left options then we write 
$\GL=\emptyset^{\ell}$ for some $\ell\in \A$ and if Right has no options then we write $\GR=\emptyset^r$ for 
some $r\in \A$.

We refer to $\emptyset^a$ as an \textit{atom} and $a\in \A$ as the \textit{adorn}. Positions in which Left (Right)
 does not have a move are called \textit{left- (right-) atomic}. A \textit{purely-atomic} position 
 is both left- and right-atomic. 
It is useful to identify  $\adr{a}=\langle \emptyset^a\mid \emptyset^a\rangle$ for any $a\in \A$. 
For example, $\adr{0}=\langle \emptyset^0\mid \emptyset^0\rangle$ where 0 is the identity of $\A$.

 \begin{definition}\label{def:recur}
 Let $\A$ be a totally ordered group and let $\Omega_0=\{\pura{\ell}{r}\mid \ell, r \in \A\}$. 
 For $n>0$, the set $\Omega_n$ is the set of all games with finite sets of options in $\Omega_{n-1}$, 
 including games which are left- and/or right-atomic, and the set of games of \textit{rank} $n$ is $\Omega_n\setminus \Omega_{n-1}$. Let $\Omega = \cup_{n\geq 0} \Omega_n$. Then $(\Omega, \A)$ is  a free \textit{space} of games.
 \end{definition}
 
 Many combinatorial games decompose into independent sub-positions as play progresses. 
A player must choose exactly one of these sub-positions and play in it. This is known as the disjunctive sum. Here, 
and elsewhere, an expression of the type $\GL + H$ denotes the list of games  of the form $G^L+H$, $G^L\in \GL$.

\begin{definition}\label{def:disjsum}
Consider a totally ordered group $\A$ and $G, H\in (\Omega,\A)$. The disjunctive sum of $G$ and $H$ is given by:
\begin{eqnarray*}
G+H&=& \langle \, \emptyset^{\ell_1+\ell_2}\mid\emptyset^{r_1+r_2} \, \rangle, \quad\textrm{ if $G=\langle \, \emptyset^{\ell_1}\mid\emptyset^{r_1} \, \rangle$ and
$H=\langle \, \emptyset^{\ell_2}\mid\emptyset^{r_2} \, \rangle$;}\\
&=&\langle \, \emptyset^{\ell_1+\ell_2}\mid\GR +H,G+\HR \, \rangle, \textrm{ if
$G=\langle \, \emptyset^{\ell_1}\mid\GR  \, \rangle$,
$H=\langle \, \emptyset^{\ell_2}\mid\HR \, \rangle$},\\
&{}&\qquad \textrm{ and at least one of $\GR $ and $\HR$
is not empty;}\\
&=&\langle \, \GL +H,G+\HL\mid \emptyset^{r_1+r_2} \, \rangle, \textrm{ if
$G=\langle \, \GL \mid\emptyset^{r_1} \, \rangle$,
$H=\langle \, \GL \mid\emptyset^{r_2} \, \rangle$},\\
&{}&\qquad \textrm{ and at least one of $\GL $ and $\HL$
is not empty;}\\
&=&\langle \, \GL +H,G+\HL\mid\GR +H,G+\HR \, \rangle,
\textrm{ otherwise.}
\end{eqnarray*}
\end{definition}

\begin{definition}\label{def:space}
A combinatorial game space is the structure 
$$\Omega=((\Omega,\A),\s, \nu_L, \nu_R,+),$$ 
where `+' is the disjunctive sum in the free space $(\Omega, \A)$, $\s$\ is a totally ordered set of game results, and $\nu_L:\A\rightarrow \s$ and $\nu_R:\A\rightarrow \s$ are order preserving maps. Moreover, if $| \A | > 1$ then require $\nu(a) = \nu_L(a) = \nu_R(a)$, for all $a\in \A$.
\end{definition}

Suppose $a,b\in \s$ with $a>b$, the standard convention is that Left prefers 
 $a$  and Right prefers $b$. The three winning conventions usually considered in the literature are:
 \begin{itemize}
\item \textit{normal-play} corresponds to: (i)  the trivial group $\A=\{0\}$ and the set $\s=\{-1,+1\}$; (ii) the maps 
$\nu_L(0)=-1$, $\nu_R(0) =+1$, 
\item \textit{mis\`{e}re-play} corresponds to:
(i) the trivial group $\A=\{0\}$ and the set $\s=\{-1,+1\}$; (ii) the maps 
$\nu_L(0) = +1$, $\nu_R(0) =-1$, 
\item \emph{scoring-play} usually corresponds to the adorns being the group of real numbers, 
with its natural order and addition, and moreover $\s= \A= \mathbb R$, and where $\nu$ is the identity map.
\end{itemize}
The \emph{conjugate} denotes the position where Left and Right have `switched roles'. 

\begin{definition}\label{def:conjugate} 
The \textit{conjugate} of $G\in \Omega$ is
\[ \overset{\leftrightarrow}{G} = 
\begin{cases}
\pura{-b}{-a}, \mbox{ if $G=\pura{a}{b}$, $a,b\in \A$}\\
\cgs{\overset{\leftrightarrow}\GR}{\emptyset^{-a} }, \mbox{ if $G=\cgs{\emptyset^{a}}{\GR}$}\\
\cgs{\emptyset^{-a}}{\overset{\leftrightarrow}\GL }, \mbox{ if $G=\cgs{\GL}{\emptyset^{a}}$}\\
\cgs{\overset{\leftrightarrow}\GR }{\overset{\leftrightarrow}\GL }, \mbox{otherwise},
\end{cases}
\]
where $\overset{\leftrightarrow}\GR$ denotes the list of games 
$\overset{\leftrightarrow}X $, for $X\in \GR$, and similarly for $\GL$. 
\end{definition}

By the recursive definition of the free space $(\Omega,\A)$, each combinatorial game space is closed under conjugation.
In normal-play, the games form an ordered group and each game $G$ has an additive inverse,
appropriately called $-G$ and $-G = \Conj{G}$. However, there are other spaces of games, for example scoring and mis\`ere 
games, where $\Conj{G}$ is not necessarily $-G$ (e.g. \cite{Mille2015}). 

 \begin{definition}\label{def:Universe}
A \textit{universe} of games, $\U\subseteq \Omega$, is a subspace of a given combinatorial game space $\Omega=((\Omega,\A),\s, \nu_L, \nu_R, +)$, with:
\begin{enumerate}
\item $\boldsymbol a=\langle \emptyset^a\mid \emptyset^a\rangle\in \U$ for all $a\in \A$;
\item \textit{options closure:}  if $A\in \U$ and $B$ is an option of $A$ then $B\in \U$;
\item  \textit{disjunctive sum closure:} if $A,B\in \U$ then $A+B\in \U$;
\item \textit{conjugate closure}: if $A\in \U$ then $\Conj{A}\in \U$;
\end{enumerate}
\end{definition}

The mapping of adorns in $\A$\ to  elements of $\s$\ is extended to positions in general via two recursively defined (optimal play) \emph{outcome functions}.
 
\begin{definition}\label{def:outcomes} 
Let $G\in \U \subseteq \Omega$ and consider given maps $\nu_L:\A\rightarrow \s$ and $\nu_R:\A\rightarrow \s$, where $S$ is a totally ordered set. The \textit{left}- and \textit{right-outcome functions} are $o_L:\Omega \rightarrow \s, o_R:\Omega\rightarrow \s$, 
where$$o_L(G) = \begin{cases} 
\nu_L(\ell) & \textrm{if  $G = \langle \emptyset^\ell\mid \GR\rangle$,} \\
 \max_L\{o_R(G^L)\} & \textrm{otherwise} 
\end{cases}
$$

$$o_R(G) = \begin{cases} 
\nu_R(r) & \textrm{if  $G = \langle \GL\mid \emptyset^r \rangle$,} \\
 \min_R\{o_L(G^R)\} & \textrm{otherwise,} 
\end{cases}
$$
where the $\max_L$ ($\min_R$) ranges over all Left (Right) options.
\end{definition}

From this we conclude that each universe is a partially ordered commutative monoid
with $\adr{0}$ as the additive identity. 

Let $G\in \U$. From Definition \ref{def:outcomes} we have that $o_L(G)=\nu_L(\ell)$ and 
$o_R(G)=\nu_R(r)$ for some $\ell,r\in\A$. Therefore we may always assume that
 the set of (left- and right-) outcomes is $\s=\{ \nu_L(a):a\in\A\}\cup\{\nu_R(a):a\in \A\}$. 

\begin{definition}
A universe $\U$ of combinatorial games is \emph{parental} if, for each pair of finite \emph{non-empty} lists,
 $\mathcal{ A}, \mathcal{ B}\subset \U$, then
 $\cg{\mathcal{ A}}{\mathcal{ B}}\in \U$.
\end{definition}

\begin{definition}
A universe $\U$ of combinatorial games is \emph{dense} if, for all $G\in \U$, for any  $x,y\in \s$, 
there is a $H\in \U$ such that $o_L(G+H)=x$ and $o_R(G+H) =y$.
\end{definition}

\begin{definition}\label{def:cgs}
A universe $\U$ of combinatorial games is an {\textit \Abs\ }if it is both parental and dense.
\end{definition}

A partial order is defined on any universe of additive combinatorial games.
\begin{definition}\label{def:order} Let $\U$ be any universe of combinatorial games. For $G, H \in \U$, $G\su H $ modulo $\U$ if and only if 
$o_L(G+ X)\ge o_L(H+ X)$ and $o_R(G+ X)\ge o_R(H+ X)$, for all games $X\in \U $.
\end{definition}
The main results for Absolute Combinatorial Game Theory \cite{LNS} are the following improvements of general game comparison. (The ``Common Normal Part'' corresponds to the Maintenance part in this paper.)

\begin{thrm}[Basic  order of games \cite{LNS}]  
Consider games $G, H\in \U$, an \Abs. Then $G\succcurlyeq H$ if and only if the following two conditions hold.\vspace{0.2 cm}

\noindent
Proviso:
\begin{enumerate}[]
 \item $o_L(G+X)\geqslant o_L(H+X)$ for all left-atomic $X\in \U$; 
  
 \item
 $o_R(G+X)\geqslant o_R(H+X)$ for all right-atomic $X\in \U$;
\end{enumerate}
\noindent
Common Normal Part:
\begin{enumerate}[]
\item For all $G^R$, there is $H^R$ such that $G^R\succcurlyeq H^R$ or there is $G^{RL}$ such that $G^{RL}\succcurlyeq H$;
\item For all $H^L$, there is $G^L$ such that $G^L\succcurlyeq H^L$ or there is $H^{LR}$ such that $G\succcurlyeq H^{LR}$.
\end{enumerate}
\end{thrm}

\begin{cor}[Subordinate game comparison \cite{LNS}]\label{thm:allcomps}  Let $G, H\in \U$, an Absolute Universe. Then $G\succcurlyeq_{\U} H$ if the Common Normal Part holds and if ${\U} $ is the
\begin{itemize}
\item  normal-play universe;
\item dicot mis\`ere-play universe, and $o(G)\geqslant o(H)$;
\item free mis\`ere-play space, and $H^\mathcal{L}=\atom{0}\Rightarrow G^\mathcal{L}=\atom{0}$ and $G^\mathcal{R}=\atom{0}\Rightarrow H^\mathcal{R}=\atom{0}$;
\item dicot scoring-play universe, and $o(G)\geqslant o(H)$;
\item guaranteed scoring-play universe, and $\underline{o}_L(G)\geqslant \underline{o}_L(H)$ and $\overline{o}_R(G)\geqslant \overline{o}_R(H)$, where $\underline{o}$ and $\overline{o}$ denotes Right's and Left's pass allowed left- and right-outcomes respectively \cite{LNS}.
\end{itemize}
\end{cor}

\section*{Appendix B}
One of the benefits of the Left Provisonal Game is that it allows for game comparison in \emph{any} \Abs\ in CG-suit. We attach code for version CG-suit 0.7 (coded by C. Santos). The procedure CompareDM requires input Left Provisonal Game as a pair of literal form (dicot mis\`ere-play) games. We begin by illustrating how to run the below code. 

\begin{verbatim}
EXAMPLE:

G=literally({0,*|*})
H=literally({0,*|{0|0,*}})
CompareM([G,H])
\end{verbatim}
\begin{verbatim}
Moutcome:=proc (G)
local a,b,c,j,w,k,l,r,i;
option remember;

l:=LeftOptions(G);
r:=RightOptions(G);
b:=Length(l);
c:=Length(r);

if (G==0) then
k:=11;
fi;

if (G!=0) then
j:=0;
for i from 1 to b do
if (Moutcome(l[i])==0 or Moutcome(l[i])==1) then j:=1; fi;
od;

w:=0;
for i from 1 to c do
if (Moutcome(r[i])==0 or Moutcome(r[i])==-1) then w:=1; fi;
od;

if (j==0 and w==0) then k:=0; fi;
if (j==0 and w==1) then k:=-1; fi;
if (j==1 and w==0) then k:=1; fi;
if (j==1 and w==1) then k:=11; fi;
fi;

return k;
end;
\end{verbatim}
\begin{verbatim}
Dual := proc (pos)
local l,r,l1,r1,l2,r2,ll1,ll2,rr1,rr2,i,aux;
option remember;

l := [];
r := [];

l1 := LeftOptions(pos[1]);
r1 := RightOptions(pos[1]);

l2 := LeftOptions(pos[2]);
r2 := RightOptions(pos[2]);

ll1:=Length(l1);
rr1:=Length(r1);
ll2:=Length(l2);
rr2:=Length(r2);

for i from 1 to rr1 do
aux:=[r1[i],pos[2]];
Add(r,Dual(aux));
od;

for i from 1 to ll2 do
aux:=[pos[1],l2[i]];
Add(r,Dual(aux));
od;

for i from 1 to ll1 do
if (Moutcome(l1[i])==1) then
aux:=[l1[i],pos[2]];
Add(l,Dual(aux));
fi;
if (Moutcome(l1[i])==11 and (Moutcome(pos[2])==11 or Moutcome(pos[2])==-1)) then
aux:=[l1[i],pos[2]];
Add(l,Dual(aux));
fi;
if (Moutcome(l1[i])==0 and (Moutcome(pos[2])==0 or Moutcome(pos[2])==-1)) then
aux:=[l1[i],pos[2]];
Add(l,Dual(aux));
fi;
if (Moutcome(l1[i])==-1 and Moutcome(pos[2])==-1)
then
aux:=[l1[i],pos[2]];
Add(l,Dual(aux));
fi;
od;


for i from 1 to rr2 do
if (Moutcome(pos[1])==1)
then
aux:=[pos[1],r2[i]];
Add(l,Dual(aux));
fi;
if (Moutcome(pos[1])==11 and (Moutcome(r2[i])==11 or Moutcome(r2[i])==-1))
then
aux:=[pos[1],r2[i]];
Add(l,Dual(aux));
fi;
if (Moutcome(pos[1])==0 and (Moutcome(r2[i])==0 or Moutcome(r2[i])==-1))
then
aux:=[pos[1],r2[i]];
Add(l,Dual(aux));
fi;
if (Moutcome(pos[1])==-1 and Moutcome(r2[i])==-1)
then
aux:=[pos[1],r2[i]];
Add(l,Dual(aux));
fi;
od;


return {l | r};
end;
\end{verbatim}
\begin{verbatim}
CompareDM := proc (pos)
local l,r,l1,r1,l2,r2,ll1,ll2,rr1,rr2,i,a,b,s;
option remember;

l := [];
r := [];

l1 := LeftOptions(pos[1]);
r1 := RightOptions(pos[1]);

l2 := LeftOptions(pos[2]);
r2 := RightOptions(pos[2]);

ll1:=Length(l1);
rr1:=Length(r1);
ll2:=Length(l2);
rr2:=Length(r2);

a:=0; b:=0;


if ((Moutcome(pos[1])==1) or (Moutcome(pos[1])==11 and 
(Moutcome(pos[2])==11 or Moutcome(pos[2])==-1)) or 
(Moutcome(pos[1])==0 and (Moutcome(pos[2])==0 or 
Moutcome(pos[2])==-1)) or (Moutcome(pos[1])==-1 and 
Moutcome(pos[2])==-1)) and (Dual(pos)>=0) then 
a:=1; 
fi;

if ((Moutcome(pos[2])==1) or (Moutcome(pos[2])==11 and 
(Moutcome(pos[1])==11 or Moutcome(pos[1])==-1)) or 
(Moutcome(pos[2])==0 and (Moutcome(pos[1])==0 or 
Moutcome(pos[1])==-1)) or (Moutcome(pos[2])==-1 and 
Moutcome(pos[1])==-1)) and (Dual([pos[2],pos[1]])>=0) then 
b:=1; 
fi;

if (a==1) and (b==1) then s:="G=H"; fi;
if (a==1) and (b==0) then s:="G>H"; fi;
if (a==0) and (b==1) then s:="G<H"; fi;
if (a==0) and (b==0) then s:="G<>H"; fi;

return s;
end;
\end{verbatim}

\bibliographystyle{plain}
\bibliography{games4}

\end{document}